\documentclass[notitlepage]{scrartcl}
\usepackage[short labels]{enumitem}
\usepackage{amsmath}
\usepackage{amssymb}
\usepackage{mathtools}

\usepackage{amsthm}
\usepackage{abstract}
\usepackage{graphicx}
\usepackage{caption}
\usepackage{subcaption}
\usepackage{xcolor}
\usepackage{comment}
\usepackage{float}
\usepackage{dsfont}
\usepackage[font=small,labelfont=bf]{caption}
\usepackage{hyperref}
\hypersetup{
    colorlinks=true,
    linkcolor=blue,
    filecolor=magenta,      
    urlcolor=cyan
    }
\makeatletter
\newcommand*{\barfix}[2][.175ex]{%
  \mathpalette{\@barfix{#1}}{#2}%
}
\newcommand*{\@barfix}[3]{%
  \vbox{%
    \kern#1\relax
    \hbox{$#2#3\m@th$}%
  }%
}
\makeatother

\newtheorem{theorem}{Theorem}
\newtheorem{thm}{Theorem}[section]
\newtheorem{corollary}[thm]{Corollary}
\newtheorem{lemma}[thm]{Lemma}
\newtheorem{claim}[thm]{Claim}
\newtheorem{remark}[thm]{Remark}
\newtheorem{question}[thm]{Question}
\newtheorem{conjecture}[thm]{Conjecture}

\usepackage[margin=1in]{geometry}

\title{A Jump of the Saturation Number \\in Random Graphs?}
\author{Sahar Diskin\footnote{School of Mathematical Sciences, Tel Aviv University, Tel Aviv 6997801, Israel.} \qquad Ilay Hoshen\footnote{School of Mathematical Sciences, Tel Aviv University, Tel Aviv 6997801, Israel.}\qquad
        Maksim Zhukovskii\footnote{Department of Computer Science, The University of Sheffield, Sheffield S1 4DP, United Kingdom.}
}

\begin{document}
\maketitle

\begin{abstract}
For graphs $G$ and $F$, the saturation number $\textit{sat}(G,F)$ is the minimum number of edges in an inclusion-maximal $F$-free subgraph of $G$. In 2017, Kor\'andi and Sudakov initiated the study of saturation in random graphs. They showed that for constant $p\in (0,1)$, \textbf{whp} $\textit{sat}\left(G(n,p),K_s\right)=\left(1+o(1)\right)n\log_{\frac{1}{1-p}}n$.

We show that for every graph $F$ and every constant $p\in (0,1)$, \textbf{whp} $\textit{sat}\left(G(n,p), F\right)=O(n\ln n)$. Furthermore, if every edge of $F$ belongs to a triangle, then the above is the right asymptotic order of magnitude, that is, \textbf{whp} $\textit{sat}\left(G(n,p),F\right)=\Theta(n\ln n)$. We further show that for a large family of graphs $\mathcal{F}$ with an edge that does not belong to a triangle, which includes all bipartite graphs, for every $F\in \mathcal{F}$ and constant $p\in(0,1)$, \textbf{whp} $\textit{sat}\left(G(n,p),F\right)=O(n)$. We conjecture that this sharp transition from $O(n)$ to $\Theta(n\ln n)$ depends only on this property, that is, that for any graph $F$ with at least one edge that does not belong to a triangle, \textbf{whp} $\textit{sat}\left(G(n,p),F\right)=O(n)$.

We further generalise the result of Kor\'andi and Sudakov, and show that for a more general family of graphs $\mathcal{F}'$, including all complete graphs $K_s$ and all complete multipartite graphs of the form $K_{1,1,s_3,\ldots, s_{\ell}}$, for every $F\in \mathcal{F}'$ and every constant $p\in(0,1)$, \textbf{whp} $\textit{sat}\left(G(n,p),F\right)=\left(1+o(1)\right)n\log_{\frac{1}{1-p}}n$. Finally, we show that for every complete multipartite graph $K_{s_1, s_2, \ldots, s_{\ell}}$ and every $p\in \left[\frac{1}{2},1\right)$, $\textit{sat}\left(G(n,p),K_{s_1,s_2,\ldots,s_{\ell}}\right)=\left(1+o(1)\right)n\log_{\frac{1}{1-p}}n$.
\end{abstract}

\section{Introduction}
\subsection{Background and main results}
For two graphs $G$ and $F$, a subgraph $H\subseteq G$ is said to be \textit{$F$-saturated} in $G$ if it is a maximal $F$-free subgraph of $G$, that is, $H$ does not contain any copy of $F$ as a subgraph, but adding any missing edge $e\in E(G)\setminus E(H)$ creates a copy of $F$ (throughout the paper, we assume for convenience that $F$ does not contain isolated vertices and note that it does not cause any loss of generality). The minimum number of edges in an $F$-saturated subgraph in $G$ is called the \textit{saturation number}, which we denote by $\textit{sat}(G,F)$. Zykov \cite{Z49}, and independently Erd\H{o}s, Hajnal, and Moon \cite{EHM64} initiated the study of the saturation number of graphs, specifically of $\textit{sat}(K_n, K_s)$. Since then, there has been an extensive study of $\textit{sat}(K_n, F)$ for different graphs $F$. Note that the maximum number of edges in an $F$-saturated graph is $ex(n,F)$, and hence the saturation problem of finding $\textit{sat}(K_n,F)$ is, in some sense, the opposite of Tur\'an problem. Of particular relevance is the following result due to K\'aszonyi and Tuza \cite{KT86}, which shows that for every graph $F$, there exists some constant $c=c(F)$ such that $\textit{sat}(K_n, F)\le cn$. We refer the reader to \cite{FFS11} for a comprehensive survey on results on saturation numbers of graphs. 

In 2017, Kor\'andi and Sudakov initiated the study of the saturation problem for random graphs, that is, when the host graph $G$ is the binomial random graph $G(n,p)$ for constant $p$. Considering the saturation number for cliques in the binomial random graph, $\textit{sat}(G(n,p), K_s)$, they showed the following:
\begin{thm}[Theorem 1.1 in \cite{KS17}]\label{th: KS main result}
    Let $0<p<1$ be a constant and let $s\ge 3$ be an integer. Then \textbf{whp}\footnote{With high probability, that is, with probability tending to $1$ as $n$ tends to infinity.}
    \begin{align*}
        \textit{sat}\left(G(n,p),K_s\right)=\left(1+o(1)\right)n\log_{\frac{1}{1-p}}n.
    \end{align*}
\end{thm}
If $F$ is a graph such that every $e\in E(F)$ belongs to a triangle, it is not hard to see that $\textit{sat}\left(G(n,p),F\right)=\Omega(n\ln n)$. Indeed, the neighbourhood $N_H(v)$ of every vertex $v$ in the saturated subgraph $H$ should dominate its neighbourhood in $G(n,p)$, and thus \textbf{whp} $|N_H(v)|$ should be at least of logarithmic order in $n$. More precisely, the lower bound in \cite{KS17} comes from the following:
\begin{thm}[Theorem 2.2 in \cite{KS17}]\label{th: KS triangleS}
    Let $0<p<1$ be a constant. Let $F$ be a graph such that every $e \in E(F)$ belongs to a triangle in $F$. Then \textbf{whp}
    \begin{align*}
        \textit{sat}\left(G(n,p), F\right)\ge \left(1+o(1)\right)n\log_{\frac{1}{1-p}}n.
    \end{align*}
\end{thm}

Following \cite{KS17}, there has been subsequent work on the saturation number $\textit{sat}\left(G(n,p), F\right)$ for other concrete graphs $F$. Mohammadian
and Tayfeh-Rezaie~\cite{MT18} and Demyanov and Zhukovskii~\cite{DZ22} proved tight asymptotics for $F=K_{1,s}$. Demidovich, Skorkin and Zhukovskii~\cite{DSZ23} proved that \textbf{whp} $\textit{sat}\left(G(n,p), C_m\right)=n+\Theta\left(\frac{n}{\ln n}\right)$ when $m\ge 5$, and showed that \textbf{whp} $\textit{sat}(G(n,p), C_4)=\Theta(n)$.

In this paper, we revisit the problem of saturation number in $G(n,p)$. Our first main result gives a general upper bound, holding for all graphs $F$.
\begin{theorem} \label{th: global 1}
Let $0<p<1$ be a constant. Let $F$ be an arbitrary graph. Then \textbf{whp} $$\textit{sat}\left(G(n,p), F\right)=O(n\ln n).$$
\end{theorem}
Comparing with the result of \cite{KT86}, Theorem \ref{th: global 1} shows that the saturation number in random graphs can be larger than the saturation number in $K_n$ by at most a factor of $O(\ln n)$, whereas Theorem \ref{th: KS triangleS} shows that this is asymptotically tight.

We note that the hidden constant in $O(n\ln n)$ which we obtain in the proof may depend on the probability $p$ and the graph $F$. Note that if every edge of $F$ belongs to a triangle in $F$, then Theorems \ref{th: global 1} and \ref{th: KS triangleS} imply that \textbf{whp} $\textit{sat}\left(G(n,p),F\right)=\Theta(n\ln n)$. In fact, we conjecture that the asymptotics of the saturation number are dictated by the assumption of Theorem~\ref{th: KS triangleS}. That is:
\begin{conjecture}\label{conjecture}
    Let $0<p<1$ be a constant. If $F$ is a graph such that \textit{every} $e\in E(F)$ belongs to a triangle in $F$, then \textbf{whp}
    \begin{align*}
        \textit{sat}\left(G(n,p), F\right)=\Theta(n\ln n).
    \end{align*}
    On the other hand, if $F$ is a graph such that there \textit{exists} $e\in E(F)$ which does \textit{not} belong to a triangle in $F$, then \textbf{whp}
    \begin{align*}
        \textit{sat}\left(G(n,p),F\right)=O(n).
    \end{align*}
\end{conjecture}
Our second main result further advances us towards settling this conjecture. We define a family of graphs for which the saturation number in $G(n,p)$ is typically linear in $n$. Given a graph $F$, let $\chi(F)$ be the chromatic number of $F$. We say that $F$ has property $(\ntriangleright)$ if there exist a largest colour class $I_{\mathrm{max}}$, among all proper colourings of $F$ with $\chi(F)$ colours, and a vertex $v \in V(F) \setminus I_{\mathrm{max}}$ such that $N_F(v) \subseteq I_{\mathrm{max}}$ (see Figure \ref{f: property ntriangle}). Note that every bipartite graph satisfies property $(\ntriangleright)$.
\begin{theorem}\label{th: global 2}
Let $0<p<1$ be a constant. Let $F$ be a graph with property $(\ntriangleright)$. Then \textbf{whp} $$\textit{sat}\left(G(n,p),F\right)=O(n).$$
\end{theorem}
Once again, the hidden constant in $O(n)$ which we obtain in the proof may depend on the probability $p$ and the graph $F$. Furthermore, observe that the vertex $v$ in Theorem \ref{th: global 2} does not belong to a triangle in $G$. In particular, since we assume $F$ has no isolated vertices and thus $v$ is not an isolated vertex, there exists at least one edge which does not lie in a triangle in $F$, supporting the second part of Conjecture \ref{conjecture}, although not resolving it completely. Furthermore, Theorem \ref{th: global 2} extends the \textit{asymptotic} results of \cite{DSZ23}, as every $F=C_m$ also satisfies property $(\ntriangleright)$.

\begin{figure}[H]
\centering
\includegraphics[width=1.8in,height=1.8in]{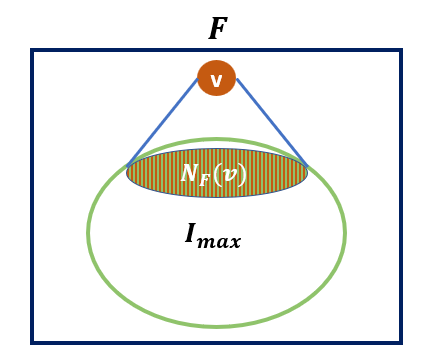}
\caption{An illustration of graphs satisfying property $(\ntriangleright)$.}
\label{f: property ntriangle}
\end{figure}

The second part of the paper aims for tight asymptotic bounds. Indeed, more ambitiously, one could try and aim for tight asymptotics in the case where every edge of $F$ belongs to a triangle (as in Theorems \ref{th: KS main result} and \ref{th: KS triangleS}). Our next two results aim at extending the tight asymptotics of Theorem \ref{th: KS main result} to a wider family of graphs. The first one extends Theorem \ref{th: KS main result} to complete multipartite graphs for $p\ge \frac{1}{2}$ (note that the case of bipartite graphs, that is $\ell=2$, is covered by Theorem \ref{th: global 2}).
\begin{theorem}\label{th: complete mp}
Let $\ell \ge 3$, $s_1\le s_2 \le \ldots \le s_{\ell}$, and let $\frac{1}{2}\le p< 1$ be a constant. Then \textbf{whp} $$\textit{sat}\left(G(n,p), K_{s_1,\ldots, s_{\ell}}\right)=\left(1+o(1)\right)n\log_{\frac{1}{1-p}}n.$$
\end{theorem}
\begin{remark}
Though it may be possible that the same result holds for constant $p<1/2$, our proof does not work in this case since one of its main ingredients, Lemma 2.2, fails to be true when $p<1/2$.
\end{remark} 

The second one defines a family of graphs for which the above is an asymptotic upper bound for all values of $p$. For two graphs $A$ and $B$, we say that a graph $B$ is \textit{$A$-degenerate}, if every two-vertex-connected subgraph of it is a subgraph of $A$. Furthermore, we say that a graph $F=(V,E)$ has the property $(\star)$ if there is an edge $\{u, v\}=e\in E$ such that for every independent set $I\subseteq V$, we have that $F[V\setminus I]$ is non-$F[V\setminus \{u,v\}]$-degenerate, that is, there exists a two-vertex-connected subgraph of $F[V\setminus I]$ which is not a subgraph of $F[V\setminus \{u,v\}]$.
\begin{theorem}\label{th: property star}
Let $0<p<1$ be a constant and let $F$ be a graph with property $(\star)$. Then \textbf{whp} $$\textit{sat}\left(G(n,p), F\right)\le \left(1+o(1)\right)n\log_{\frac{1}{1-p}}n.$$
\end{theorem}

Let us mention a family of graphs satisfying property $(\star)$, which could be of particular interest. Let $F$ be a graph with an edge $\{u, v\}$ such that there exists a proper colouring of $F$ with $\chi(F)$ colours, where $\{u\}$ and $\{v\}$ are distinct colour classes (see Figure \ref{f: property star}), and let us further suppose that for every independent set $I\subseteq V(F)$, $F[V(F)\setminus I]$ is two-vertex-connected. Then $F\setminus \{u,v\}$ is $\chi(F)-2$ colourable, whereas for any independent set $I$, $F[V\setminus I]$ requires at least $\chi(F)-1$ colours, and therefore $F$ satisfies property $(\star)$.
\begin{figure}[H]
\centering
\includegraphics[width=1.7in,height=1.7in]{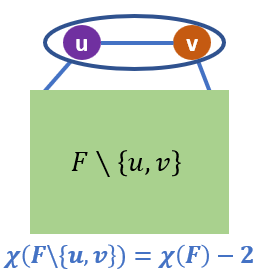}
\caption{An illustration of the edge $\{u,v\}$ and the remaining graph $F\setminus\{u,v\}$.}
\label{f: property star}
\end{figure}
Note that if, in addition, every edge of $F$ lies in a triangle, then by Theorem \ref{th: KS triangleS}, we obtain sharp asymptotics. That is, \textbf{whp} $\textit{sat}\left(G(n,p),F\right)=\left(1+o(1)\right)n\log_{\frac{1}{1-p}}n$. In particular, we obtain the following corollary.
\begin{corollary}\label{c: immediate}
    Let $0<p<1$ be a constant. Let $\ell\ge3$ and $s_3,\ldots, s_{\ell}\geq 1$ be integers. Let $F=K_{1,1,s_3,\ldots, s_{\ell}}$. Then \textbf{whp}
    \begin{align*}
        \textit{sat}\left(G(n,p),F\right)=\left(1+o(1)\right)n\log_{\frac{1}{1-p}}n.
    \end{align*}
\end{corollary}
Observe that Theorem \ref{th: complete mp} provides the same asymptotic as Corollary \ref{c: immediate} for any multipartite graph $F$, that is, without the requirement that $s_1=s_2=1$, however only for $p\in \left[\frac{1}{2},1\right)$.

Theorem \ref{th: KS triangleS} together with Theorems \ref{th: complete mp} and \ref{th: property star} suggest that one may make a more ambitious claim in Conjecture \ref{conjecture}, stating that if every edge of $F$ lies in a triangle, then \textbf{whp} $\textit{sat}\left(G(n,p),F\right)=\left(1+o(1)\right)n\log_{\frac{1}{1-p}}n$ (see more on that in Section \ref{s: discussion}).

\subsection{Main methods and organisation}
Let us begin with some conventions that will be useful for us throughout the paper. Given two graphs $H\subseteq G$, and a graph $F$, we say that an edge $e\in E(G)$ is \textit{completed by} $H$ (or that $H$ completes $e$), if there is a subgraph of $H'\subseteq H$ such that $e\cup H'\cong F$. In this case, we also say that $e$ \textit{completes} a copy of $F$. Furthermore, we use the notion of an $\mathcal{F}$\textit{-saturated} graph for a family of graphs $\mathcal{F}$. We say that a graph $H$ is $\mathcal{F}$-saturated in $G$ if $H$ does not contain a copy of \textit{any} $F\in\mathcal{F}$, yet adding an edge $e\in E(G)\setminus E(H)$ closes a copy of\textit{ at least one} $F\in\mathcal{F}$.

The structure of the paper is as follows. In Section \ref{s: lemmas} we present and establish several lemmas which we will use throughout the paper. 

In Section \ref{s: global} we prove Theorems \ref{th: global 1} and \ref{th: global 2}. Both proofs rely on a construction of an $\mathcal{F}$-saturated subgraph of $G(n,p)$ given in the proof of Lemma \ref{l: bipartite}. The proof of Theorem \ref{th: global 2} follows rather immediately from the above-mentioned construction, whereas the proof of Theorem \ref{th: global 1} utilises an inductive argument based on Lemma \ref{l: bipartite}.

In Section \ref{s: sharp} we prove Theorems \ref{th: complete mp} and \ref{th: property star}. The proof of Theorem \ref{th: property star} utilises the construction of \cite{KS17}, refined with the following lemma. For every two graphs $G$ and $A$, and $\varepsilon>0$, we say that a graph $G$ is $\varepsilon$\textit{-dense} with respect to $A$ if every induced subgraph of $G$ on at least $\varepsilon|V(G)|$ vertices contains a copy of $A$. Finally, given a family of graphs $\mathcal{F}$, we say that $G$ is $\mathcal{F}$-free if it does not contain a subgraph isomorphic to $F$ for every $F\in \mathcal{F}$. We will make use of the following result:
\begin{lemma}\cite[Theorem 2.1 and Remark 2]{DHKZ23}\label{l: ABlemma}
Let $p\in (0,1)$. Let $A$ be a graph and let $\mathcal{F}$ be a family of graphs such that every $F\in \mathcal{F}$ is non-$A$-degenerate. Then for every sufficiently small $\delta>0$, \textbf{whp} there is a spanning subgraph in $G(n,p)$ which is $\mathcal{F}$-free and $n^{-\delta}$-dense with respect to $A$.
\end{lemma}
We note that in \cite{DHKZ23} it was shown that the condition on $\mathcal{F}$ in the above lemma is not only sufficient, but also necessary. In particular, this implies some limitations in the results that can be obtained using the construction of \cite{KS17} (see Section \ref{s: discussion} for more details on the matter). 

The proof of Theorem \ref{th: complete mp} is more delicate, and is the most involved in this paper. Let us briefly outline the key ideas in the construction, while comparing them with the construction of \cite{KS17} for cliques. 

Very roughly, in \cite{KS17}, one takes a set $A$ on $\Theta(\log n)$ vertices, and let $B\coloneqq [n]\setminus A$. One can then find in $A$ a spanning subgraph of $G(n,p)$ that is $K_{s-1}$-free, but $(1/\ln^3|A|)$-dense with respect to $K_{s-2}$, denote this subgraph by $A'$. Let $H$ be the subgraph of $G(n,p)$ with all the edges between $A$ and $B$, all the edges of the subgraph $A'$, and no edges inside $B$. Then, typically almost any pair of vertices $u, v$ in $B$ is likely to have many common neighbours in $A$ and one can find a copy of $K_{s-2}$ in $A'$ induced by this common neighbourhood, such that together with the edge $\{u,v\}$ they close a copy of $K_{s}$, while on the other hand the graph is $K_{s}$-free since $A'$ is $K_{s-1}$-free and $B$ is empty. 

Thus, in the case of a complete graph, one is mainly concerned with the property that the endpoints of every edge in $G(n,p)[B]$ have a large common neighbourhood in $A$. However, in the case of a complete $r$-partite graph $F$ with all parts being non-trivial, one cannot assume that in every large induced subgraph of $A'$ there is a copy of $F$ minus an edge, as otherwise it is easy to see that our graph is not $F$-free. In particular, one cannot consider an empty graph in $B$. Moreover, a suitable subgraph that we take in $B$ should satisfy the property that every vertex and its $B$-neighbours have a large common neighbourhood in $A$. In order to show the likely existence of such a subgraph in $B$, we use a coupling with an auxiliary random graph in a Hamming space and prove a tight bound for its independence number using a covering-balls argument (see Lemma \ref{claim:bound-on-alpha-G_W} and the paragraph after its statement).

Finally, in Section \ref{s: discussion} we discuss the obtained results and some of their limitations, and mention some questions as well as open problems.

\section{Preliminary lemmas}\label{s: lemmas}
Given a graph $H$ and a vertex $v\in V(H)$, we denote by $N_H(v)$ the neighbourhood of $v$ in $H$ and by $d_H(v)=|N_H(v)|$. Given a subset $S\subseteq V(H)$, we denote by $N_H(S)$ the common neighbourhood of all $v\in S$ in $H$ and by $d_H(S)=|N_H(S)|$. That is, $N_H(S)\coloneqq \bigcap_{v\in S}N_H(v)$. Finally, given subsets $S_1,S_2\subseteq V(H)$, we denote by $N_H(S_1| S_2)$ the common neighbourhood of $S_1$ in $S_2$ in the graph $H$ and by $d_H(S_1|S_2)=|N_H(S_1|S_2)|$. That is, $N_H(S_1| S_2)\coloneqq N_H(S_1)\cap S_2$. When the graph $H$ is clear from context, we may omit the subscript. We denote by $K_s^{\ell}$ the complete $\ell$-partite graph where each part is of size $s$. We omit rounding signs for the sake of clarity of presentation.

We will make use of the following bounds on the tail of binomial distribution (see, for example, \cite[Theorem 2.1]{JLR00}, \cite[Theorem A.1.12]{AS16}, and \cite[Claim 2.1]{KS17}).
\begin{lemma}\label{l: chernoff}
Let $N\in \mathbb{N}$, $p\in [0,1]$, and $X\sim Bin(N,p)$. Then, for $0\le a \le Np$ and for $b\ge 0$,
\begin{align}
    &\mathbb{P}\left(|X-Np|\ge a\right)\le 2\exp\left(-\frac{a^2}{3Np}\right),\label{l: chernoff 1}\\
    &\mathbb{P}(X > b Np) \le \left(\frac{e}{b}\right)^{b N p},\label{l: chernoff 2}\\
    &\mathbb{P}\left(X\le \frac{N}{\ln^2N}\right)\le (1-p)^{N-\frac{N}{\ln N}}.\label{l: chernoff 3}
\end{align}
\end{lemma}

We will require the following probabilistic lemma which shows that large enough sets are very likely to have a vertex whose number of neighbours in this set deviates largely from the expectation.

For every $p\in (0,1)$, we set
\begin{align}
    \rho(p)\coloneqq \frac{1}{1-p}.
\end{align}
When the choice of $p$ is clear, we may abbreviate $\rho\coloneqq \rho(p)$.
\begin{lemma}\label{l: many neighbours}
Let $p\in \left[\frac{1}{2}, 1\right)$, and let $G\sim G(n,p)$. For every $\gamma\ge 0$, there exists a sufficiently small $\epsilon>0$ such that the following holds. Let $X$ and $Y$ be disjoint sets of vertices in $V(G)$ of sizes $(1+\epsilon)\log_{\rho}n$ and at least $\frac{n}{\ln^3n}$, respectively. Then, 
\begin{align*}
    \mathbb{P}\left(\forall y\in Y\ d(y|X) < \left(1+(1-\gamma)\epsilon\right)\log_{\rho}n\right)\le \exp\left(-n^{\epsilon}\right).
\end{align*}
\end{lemma}
\begin{proof}
Note that for every vertex $y\in Y$, the number of neighbours of $y$ in $X$ in the graph $G$ is distributed according to $Bin(|X|,p)$. Hence, for a fixed $y\in Y$,
\begin{align*}
    \mathbb{P}\Big(d(y|X) \ge  &\left(1+(1-\gamma)\epsilon\right)\log_{\rho}n\Big)\ge \mathbb{P}\left(d(y|X)=\left(1+(1-\gamma)\epsilon\right)\log_{\rho}n\right)\\
    &=\binom{|X|}{\left(1+(1-\gamma)\epsilon\right)\log_{\rho}n}p^{\left(1+(1-\gamma)\epsilon\right)\log_{\rho}n}(1-p)^{\gamma\epsilon\log_{\rho}n}\\
    &\ge \left(\frac{(1+\epsilon)(1-p)}{\gamma\epsilon}\right)^{\gamma\epsilon\log_{\rho}n}p^{\left(1+(1-\gamma)\epsilon\right)\log_{\rho}n}\\
    &\ge (1-p)^{\gamma \epsilon \log_{1-p} \left(\frac{1+\epsilon}{\gamma \epsilon}\right) \log_{\rho} n + \gamma\epsilon\log_{\rho}n} (1-p)^{ (\log_{1-p} p) \left(1+(1-\gamma)\epsilon\right)\log_{\rho}n} \\
    &= n^{-\gamma \epsilon \log_{1-p} \left(\frac{1+\epsilon}{\gamma \epsilon}\right) - (\log_{1-p} p) \left(1+(1-\gamma)\epsilon\right) - \gamma \epsilon}.
\end{align*}

Since $p \ge \frac{1}{2}$, we have that $\log_{1-p} p \le 1$ and thus
\begin{align*}
    \mathbb{P}\Big(d(y|X) \ge  \left(1+(1-\gamma)\epsilon\right)\log_{\rho}n\Big) 
    &\ge n^{-\gamma \epsilon \log_{1-p} \left(\frac{1+\epsilon}{\gamma \epsilon}\right) - 1 - \epsilon} \\
    &\ge n^{-1 + 2\epsilon},
\end{align*}
where the last inequality is true since for sufficiently small $\epsilon$ we have $-\gamma \epsilon \log_{1-p} \left(\frac{1+\epsilon}{\gamma \epsilon}\right) \ge 3\epsilon$.

Therefore, since $d(y|X)$ are independent for $y\in Y$, we obtain that
\begin{align*}
    \mathbb{P}\left(\forall y\in Y\ d(y|X)\le \left(1+(1-\gamma)\epsilon\right)\log_{\rho}n\right)&\le\left(1-n^{-1+2\epsilon}\right)^{|Y|}\\
    &\le \exp\left(-n^{\epsilon}\right),
\end{align*}
where in the last inequality we used our assumption that $|Y|\ge \frac{n}{\ln^3n}$.
\end{proof}

We will also utilise the fact that random graphs typically have relatively small chromatic number (see, for example, Chapter 7 in \cite{FK16}):
\begin{lemma}\label{l: chromatic number}
    Let $0<p<1$ be a constant. Then \textbf{whp} $\chi(G(n,p))=O\left(\frac{n}{\ln n}\right)$.
\end{lemma}

\section{Global bounds}\label{s: global}
We begin by giving a construction showing that, for any family of graphs which contains at least one bipartite graph, \textbf{whp} the saturation number in $G(n,p)$ is linear in $n$. This construction will be key for both the proofs of Theorems \ref{th: global 1} and \ref{th: global 2}.
\begin{lemma}\label{l: bipartite}
Let $p\in(0,1)$. Let $\mathcal{F}$ be a family of graphs which contains at least one bipartite graph. Then \textbf{whp}
\begin{align*}
    \textit{sat}\left(G(n,p),\mathcal{F}\right)=O(n).
\end{align*}
\end{lemma}
\begin{proof}
Set
\begin{align*}
    \ell\coloneqq \min\left\{|V_1|\colon F\in\mathcal{F}, \chi(F)=2, V_1\text{ and } V_2 \text{ are colour classes of } F \text{ and } V(F)=V_1\sqcup V_2\right\}.
\end{align*}
We may assume that $\ell\ge 2$, as the case of star is immediate --- an $\mathcal{F}$-saturated graph then has bounded maximum degree. Let us further denote by $F_0$ a bipartite graph from $\mathcal{F}$ which achieves the above $\ell$.

Let $\tau$ be the smallest integer satisfying
\[
    (1-p^{\ell - 1})^\tau n \le n^{2/5}.
\]

We construct a subgraph $H\subseteq G(n,p)$ which is $\mathcal{F}$-saturated such that \textbf{whp} $|E(H)|=O(n)$ in two stages. Let us fix $\tau$ vertex disjoint sets of size $\ell-1$, denote them by $A_1,\ldots, A_{\tau}$ and set $A=\bigcup_i A_i$. We then proceed iteratively. At the first step, we set $B_1$ to be the set of all common neighbours of $A_1$ in $G(n,p)$ among the vertices outside $A$. At the $i$-th step, where $1<i\le \tau$, we set $B_i$ to be the set of all common neighbours of $A_i$ in $G(n,p)$ among the vertices outside $A\cup\bigcup_{j<i}B_j$. The subgraph $H$ contains all the edges between $A_i$ and $B_i$ in $G(n,p)$ for every $1\le i\le \tau$. 

Observe that $H$ is $\mathcal{F}$-free since it contains only bipartite graphs admitting a $2$-colouring with one colour class of size at most $\ell-1$. Moreover, every edge in $H$ is incident to $B_i$, for some $1\le i\le \tau$. Hence, the number of edges in $H$ thus far is at most
\[
    \sum_{i} |B_i| (\ell - 1) \le (\ell - 1)n.
\]

We now turn to add edges to $H$ such that it becomes $\mathcal{F}$-saturated. First, we consider edges whose both endpoints are in $B_i$, for some $1\le i\le \tau$. For every $1\le i\le \tau$, as long as there is an edge in $G(n,p)[B_i]$ which does not close a copy of $F \in \mathcal{F}$, we add it to $H$. Note that the degree of every vertex increased by at most $|F_0| - \ell - 1$. Indeed, if for some $1\le i\le \tau$ there is a vertex $v \in B_i$ with degree $|F_0| - \ell$ in $H[B_i]$, then we can form a copy of $K_{\ell, |F_0| - \ell}$ with $v, N_{H[B_i]}(v)$, and $A_i$. However, $F_0 \subseteq K_{\ell, |F_0| - \ell}$, a contradiction. Hence, the number of edges that are added to $H$ is at most
\[
    \sum_{i} |B_i| (\ell - 1 + |F_0| - \ell) \le (|F_0| - 1)n.
\]

Now, for every $1\le i\le \tau$, as long as there is an edge between $[n] \setminus \cup_{j \le i} B_j$ and $B_i$ which does not close a copy of $F \in \mathcal{F}$, we add it to $H$. For every vertex $v\notin A$, the probability that $v \notin \cup_{j=1}^{i} B_j$ is $(1-p^{\ell - 1})^i$. By Lemma \ref{l: chernoff}\eqref{l: chernoff 1}, the probability that there are at least $2(1-p^{\ell - 1})^i n$ such vertices is at most $\exp\left(-n^{1/4}\right)$. Thus, by the union bound over all less than $n$ choices of $i \le \tau$, \textbf{whp}
\begin{align}
    n - \sum_{j=1}^{i}|B_j|\le 2(1-p^{\ell-1})^i n, \quad \forall 1\le i\le \tau. \label{eq: outside B}
\end{align}
Hence, by \eqref{eq: outside B}, the number of edges we add to $H$ in this step is \textbf{whp} at most
\begin{align*}
    \sum_{i \le \tau} \left(n - \sum_{j=1}^{i}|B_j|\right)(|F_0| - \ell - 1) \le \sum_{i} 2(1-p^{\ell - 1})^i n(|F_0| - \ell - 1) = O(n).
\end{align*}

Furthermore, as long as there is an edge in $G(n,p)$ outside $A$ which closes a copy of some $F\in\mathcal{F}$, we add it to $H$. Note that, at this step, we only add edges induced by $[n]\setminus\left(A \cup \bigcup_i B_i\right)$. By \eqref{eq: outside B} (for $i=\tau$) and our choice of $\tau$, we have that \textbf{whp}  $\left|[n]\setminus\left(A \cup \bigcup_{j=1}^{\tau} B_i\right)\right| \le 2(1-p^{\ell - 1})^\tau n = o(\sqrt{n})$. Therefore, we only add $o(n)$ many edges in this step.

Finally, we are left with considering edges induced by $A$ and between $A$ and $[n]\setminus \left(A\cup\bigcup_iB_i\right)$. Since $|A|=O(\ln n)$, $A$ induces at most $O\left(\ln^2n\right)$ edges, and by the same reasoning as above there are \textbf{whp} at most $o(\ln n\sqrt{n})$ edges between $A$ and $[n]\setminus\left(A\cup\bigcup_iB_i\right)$. Hence, we can add to $H$ these edges one by one, until the resulting graph will be $\mathcal{F}$-saturated. We thus obtain the required subgraph $H$ where \textbf{whp} $|E(H)|=O(n)$.
\end{proof}

The proof of Theorem \ref{th: global 2} follows a similar construction to the one in Lemma \ref{l: bipartite}.
\begin{proof}[Proof of Theorem \ref{th: global 2}]
Let $F$ be a graph and let $I_{\mathrm{max}} \subseteq V(F)$ be a colour class of maximum size among all proper colourings of $F$ with $\chi(F)$ colours. We further assume that there is a vertex $v \in V(F) \setminus I_{\mathrm{max}}$ such that $N_F(v) \subseteq I_{\mathrm{max}}$. We repeat the construction of Lemma \ref{l: bipartite} with the following difference: we require each of the sets $A_i$ to induce a copy $F \setminus (I_{\mathrm{max}} \cup \{v\})$. Note that this is indeed possible as we may first take a subset $U\subset[n]$ of $n^\epsilon$ vertices, for some small $\epsilon > 0$, and \textbf{whp} we can find a $K_{|V(F)|}$-factor (see, for example, \cite{JKV08}), and then we can take a copy of $F \setminus (I_{\mathrm{max}} \cup \{v\})$ from the first $\tau$ cliques. Then $B_i$ should be selected from $[n]\setminus U$.

Note that, for any $i\in [\tau]$, $A_i$ and its common neighbourhood $B_i$ do not contain a copy of $F$. Indeed, if we had a copy of $F$, denote it by $\Tilde{F}$, then $\chi(\Tilde{F}[A_i\cap V(\tilde{F})])\le \chi(\tilde{F}[A_i]) =\chi(F)-1$, and thus we could colour $V(\tilde{F}[B_i\cap V(\tilde{F})])$ with one colour and $V(\Tilde{F}[A_i\cap V(\tilde{F})])$ with $\chi(F)-1$ colours, thus obtaining a colour class of size $|V(\Tilde{F}[B_i\cap V(\tilde{F})])| \ge |I_{\mathrm{max}}| + 1$ --- a contradiction to the assumption that $I_{\mathrm{max}}$ is a colour class of $F$ of maximum size.

Observe that connecting any $v\notin A_i$ to $|I_{\mathrm{max}}|$ vertices in $B_i$ creates a copy of $F$ with $v, N_{H[B_i]}(v)$ and $A_i$. Thus, by the same arguments as in Lemma \ref{l: bipartite}, the obtained graph $H$ is $F$-saturated, and has \textbf{whp} $O(n)$ edges.
\end{proof}

Utilising Lemma \ref{l: bipartite}, we can now prove Theorem \ref{th: global 1}.
\begin{proof}[Proof of Theorem \ref{th: global 1}]
In fact, we will prove a slightly stronger statement: that for any finite family of graphs $\mathcal{F}$, we have that \textbf{whp} $\textit{sat}\left(G(n,p),\mathcal{F}\right)=O(n\log n)$ for any fixed $p\in (0,1)$ --- as we prove by induction, this will help us with the inductive step.

Set $\chi_0\coloneqq \chi(\mathcal{F})=\min_{F\in \mathcal{F}}\chi(F)$. We prove by induction on $\chi_0$. If $\chi_0=2$, we are done by Lemma \ref{l: bipartite}. We may assume now that $\chi_0\ge 3$, and that the statement holds for any family $\mathcal{F}'$ with $\chi(\mathcal{F}')<\chi_0$.

We now construct $H\subseteq G(n,p)$ such that \textbf{whp} $H$ is $\mathcal{F}$-saturated, and $e(H)=O(n\ln n)$. We begin by letting $H$ be the empty graph. Let $A$ be a set of $C\ln n$ vertices for some large enough constant $C\coloneqq C(\mathcal{F},p)>0$. Set
\begin{align*}
    \hat{\mathcal{F}}\coloneqq \left\{F\setminus I\colon F\in \mathcal{F} \text{ and } I \text{ is an independent set of } F\right\}.
\end{align*}
We thus have that for some $F'\in \hat{\mathcal{F}}$, $\chi(F')=\chi_0-1$. By induction, \textbf{whp} there exists $H'\subseteq [n]\setminus A$ which is $\hat{\mathcal{F}}$-saturated in $G(n,p)[[n]\setminus A]$ with $O(n\ln n)$ edges. Furthermore, for a fixed set $X$ of order $\max_{F\in\mathcal{F}}|V(F)|$ in $[n]\setminus A$, $d_{G(n,p)}(X|A)$ is distributed according to $Bin\left(|A|,p^{|X|}\right)$. Thus, by Lemma \ref{l: chernoff}\eqref{l: chernoff 1},
\begin{align*}
    \mathbb{P}\left(d_{G(n,p)}(X|A)\le \max_{F\in\mathcal{F}}|V(F)|\right)\le \exp\left(-\frac{C(\ln n)p^{|X|}}{4}\right)\le n^{-2\max_{F\in\mathcal{F}}|V(F)|},
\end{align*}
for $C$ large enough. Thus, by the union bound, \textbf{whp} for every set of order $\max_{F\in\mathcal{F}}|V(F)|$ in $[n]\setminus A$, we can find a set of $\max_{F\in\mathcal{F}}|V(F)|$ common neighbours in $A$. Therefore, when we add a missing edge from $G(n,p)[V\setminus A]$, we close a copy of some $F'\in\hat{\mathcal{F}}$. This copy of $F'$ together with its common neighbours in $A$ form a copy of some $F\in\mathcal{F}$ (see Figure \ref{f: completing F}).

Let $E'$ be the set of edges of $G(n,p)$ between $[n]\setminus A$ and $A$. Let $H$ be $H'$ together with $E'$. 

First, note that $H$ is $\mathcal{F}$-free. Indeed, since $H[A]$ is an empty graph, if there is a copy of $F$ in $H$, $F[A]$ must be an independent set. However, by definition of $H'$, $F[[n]\setminus A]$ is free of any $F\setminus I$ for any independent set $I$ of $F$ --- a contradiction. 

Furthermore, by construction, every edge of $G(n,p)\setminus E(H)$ in $[n]\setminus A$ closes a copy of $F\in\mathcal{F}$. Since $|E'|=O(n\ln n)$ and \textbf{whp} $|E(H')|=O(n\ln n)$, we have that \textbf{whp} $|E(H)|=O(n\ln n)$. In order to ensure that $H$ is $\mathcal{F}$-saturated, the only edges to consider are those where both endpoints are in $A$, and there are at most $O(\ln^2n)$ such edges. Hence, \textbf{whp} there exists $H$ which is $\mathcal{F}$-saturated with $O(n\ln n)$ edges. 
\end{proof}
\begin{figure}[H]
\centering
\includegraphics[width=0.4\linewidth]{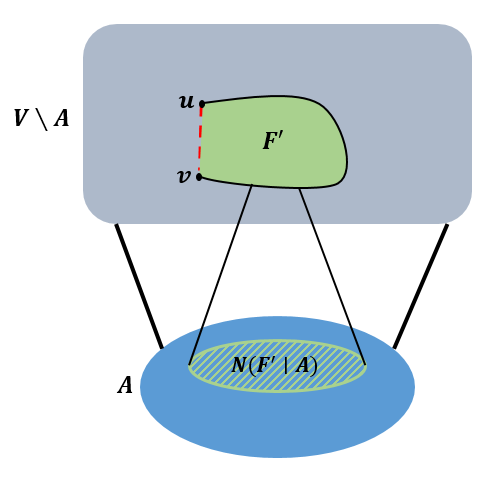}
\caption{In dashed red line there is a missing edge $\{u,v\}$ which closes a copy of $F'\in \hat{\mathcal{F}}$. Together with its common neighbourhood in $A$ (coloured light green and blue), this closes a copy of $F\in\mathcal{F}$.}
\label{f: completing F}
\end{figure}

\section{Sharp bounds}\label{s: sharp}
In this section we prove Theorems \ref{th: complete mp} and \ref{th: property star}. We begin with the proof of Theorem \ref{th: complete mp}. 

Let us begin with an outline of the proof. We show that \textbf{whp} there exists a subgraph $H \subseteq G(n,p)$ which is $K_{s_1, \dots, s_{\ell}}$-saturated and $e(H) \le (1+o(1)) n \log_{\frac{1}{1-p}} n$. Note that if we find a subgraph $H \subseteq G(n,p)$, with $e(H) \le (1+o(1)) n \log_{\frac{1}{1-p}} n$, which completes all but at most $o(n \ln n)$ edges, then we can add edges one by one if necessary, and obtain a subgraph which is $K_{s_1, \dots, s_{\ell}}$-saturated and has at most $(1+o(1)) n \log_{\frac{1}{1-p}} n$ many edges.

Very roughly, we take a subset $A$ of order $\Theta(\ln n)$ and $B$ from $[n]\setminus A$. We then find a subgraph in $G(n,p)[A]$ which is $\left\{K_{\ell}, K_{s_1}^{(\ell - 1)}\right\}$-free and such that there is a copy of $K_{s_1 - 1, s_3, \dots, s_\ell}$ in every large enough subset of $A$ (where if $s_1=1$, $K_{s_1-1,s_3,\dots,s_{\ell}}=K_{s_3,\dots,s_{\ell}}$). In the case of $s_\ell = 1$, as in \cite{KS17}, it suffices to set $B$ as the empty graph and draw all the edges between $A$ and $B$, and then to show that almost all the edges are completed. However, as in \cite{KS17}, there will still be a small (yet not negligible) amount of edges that will not be completed, as the co-degree of their endpoints in $A$ is too small. For these type of edges, some additional technical work is required, which will force us to maintain two additional small sets outside of $B$ --- $A_2$ and $A_3$ --- which, as in \cite{KS17}, will allow us to deal with these problematic edges. The case of $s_\ell \ge 2$ is naturally more delicate, as $B$ cannot be taken to be an empty graph, but instead requires some special properties. Using a novel construction, we find a subgraph in $G(n,p)[B]$ which is almost-$(s_2-1)$-regular (that is, almost all its vertices are of degree $s_2-1$, and the others might have smaller degree) and $K_{s_1,s_2-s_1+1}$-free in $G(n,p)[B]$. This subgraph will have another crucial property -- the vertices of any copy of $K_{1,s_2-1}$ have a large common neighbourhood in $A$, which we show through coupling and covering-balls arguments (the construction of this graph is the most involved part of the proof and includes key new ideas, and appears in Lemma \ref{lemma:H_B_1-subgraph}). In this way, almost all edges in $B$ close a copy of $K_{1, s_2}$ such that this copy has a large common neighbourhood in $A$, in which we can find a copy of $K_{s_1 - 1, s_3, \dots, s_\ell}$ (as in the clique case, some additional technical work is required to deal with the other edges). These two copies form a copy of $K_{s_1, \dots, s_\ell}$ as needed. 

Note that the requirement that the subgraph in $G(n,p)[B]$ is $K_{s_1,s_2-s_1+1}$-free is necessary, as otherwise a copy of $K_{s_1,\ldots, s_{\ell}}$ could be formed when drawing the edges between $B$ and $A$.

\subsection{Proof of Theorem \ref{th: complete mp}}
We may assume that $s_{\ell}\ge 2$, as the case of cliques has been dealt with in \cite{KS17}. 

Let $\gamma,\epsilon >0$ be sufficiently small constants. Let $G\sim G(n,p)$. Let $L\coloneqq L\left(s_1,\ldots, s_{\ell}\right)$ be a constant large enough with respect to $s_1,\ldots, s_{\ell}$. Let $\rho=\rho(p)$ and set
\begin{align*}
    a_1 = \frac{1}{p}\left(1 + (1+\gamma)\epsilon\right) \log_{\rho} n,\quad a_2 = L \log_{\rho}n,\quad a_3 = \frac{a_2}{\sqrt{\ln a_2}}.
\end{align*}

Let $A_1$, $A_2$, and $A_3$, be disjoint sets of vertices of sizes $a_1$, $a_2$, and $a_3$, respectively. Set $B \coloneqq V \setminus (A_1 \cup A_2 \cup A_3)$. Set
\[
    I \coloneqq \Big[\left(1 + \epsilon\right) \log_{\rho} n,\ \left(1 + (1+2\gamma)\epsilon\right) \log_{\rho} n\Big].
\]

We say that a vertex $v \in B$ is $A_1$-\emph{good} if $d_{G}(v| A_1) \in I$. Otherwise, we say that $v$ is $A_1$-\emph{bad}. Let $B_1 \subseteq B$ be the set of $A_1$-good vertices, and set $B_2 = B \setminus B_1$. Note that \textbf{whp} $|B_2|=O\left(\frac{n}{\ln n}\right)$. Indeed, for every vertex $v \in B$, we have $d_{G}(v| A_1) \sim \mathrm{Bin}(a_1, p)$. By Lemma \ref{l: chernoff}\eqref{l: chernoff 1}, for every vertex $v \in B$,
\[
    \mathbb{P}(v\ \text{is $A_1$-bad}) \le \exp\left(-c\ln n\right),
\]
for some constant $c > 0$. Thus, $\mathbb{E}[|B_2|] = O\left(n^{1-c}\right)$. By Markov's inequality, \textbf{whp} $|B_2| = O\left(\frac{n}{\log_{\rho}n}\right)$.

As mentioned prior to the proof, a key element in the proof is finding in $B_1$ a subgraph $H_{B_1}$ of $G$ which is $K_{s_1, s_2 - s_1 + 1}$-free and almost-$(s_2-1)$-regular in the edges in $G[B_1]$. Moreover, we want $H_{B_1}$ to satisfy the following property. Every vertex $v \in B_1$ and its neighbours in $H_{B_1}$ have a large common neighbourhood in $A_1$ in $G$. The proof of this key lemma is deferred to the end of this proof.
\begin{lemma}\label{lemma:H_B_1-subgraph}
    \textbf{Whp} there exists $H_{B_1} \subseteq G$ with $V(H_{B_1})=B_1$ such that the following holds.
    \begin{itemize}
        \item $H_{B_1}$ is $K_{s_1, s_2-s_1+1}$-free.
        \item The maximum degree of $H_{B_1}$ is $s_2-1$, and all but $O\left(\frac{n}{\log n}\right)$ of its vertices are of degree $s_2-1$ in $H_{B_1}$.
        \item For every $u,v\in V(H_{B_1})$ such that $\{u,v\}\in E(H_{B_1})$, $d_{G}(u, v| A_1) \ge (1 + (1-6\gamma)\epsilon) \log_{\rho} n$.
    \end{itemize} 
\end{lemma}

Since $K_{\ell}$ and $K_{s_1}^{(\ell-1)}$ are two-vertex-connected graphs (as $\ell\ge 3$ by assumption) and are non-$K_{s_1-1,s_3,\ldots,s_{\ell}}$-degenerate, by Lemma \ref{l: ABlemma}, we can take a spanning subgraph $H_{A_1} \subseteq G[A_1]$ which is $(1/\ln^3|A_1|)$-dense with respect to $K_{s_1 - 1, s_3, \dots, s_{\ell}}$ and $\{K_{\ell}, K_{s_1}^{(\ell-1)}\}$-free. Moreover, let us move the vertices which are of degree less than $s_2-1$ in $H_{B_1}$ from $B_1$ to $B_2$. Note that by Lemma~\ref{lemma:H_B_1-subgraph}, \textbf{whp} we moved $O\left(\frac{n}{\log n}\right)$ vertices from $B_1$ to $B_2$. 

Let $H_1$ be the graph on $A_1 \cup B$ with the edges $H_{A_1}$, $H_{B_1}$, and all the edges between $A_1$ and $B$ in $G$. We now continue with a series of self-contained claims. 
\begin{figure}[H]
\centering
\includegraphics[width=0.8\linewidth]{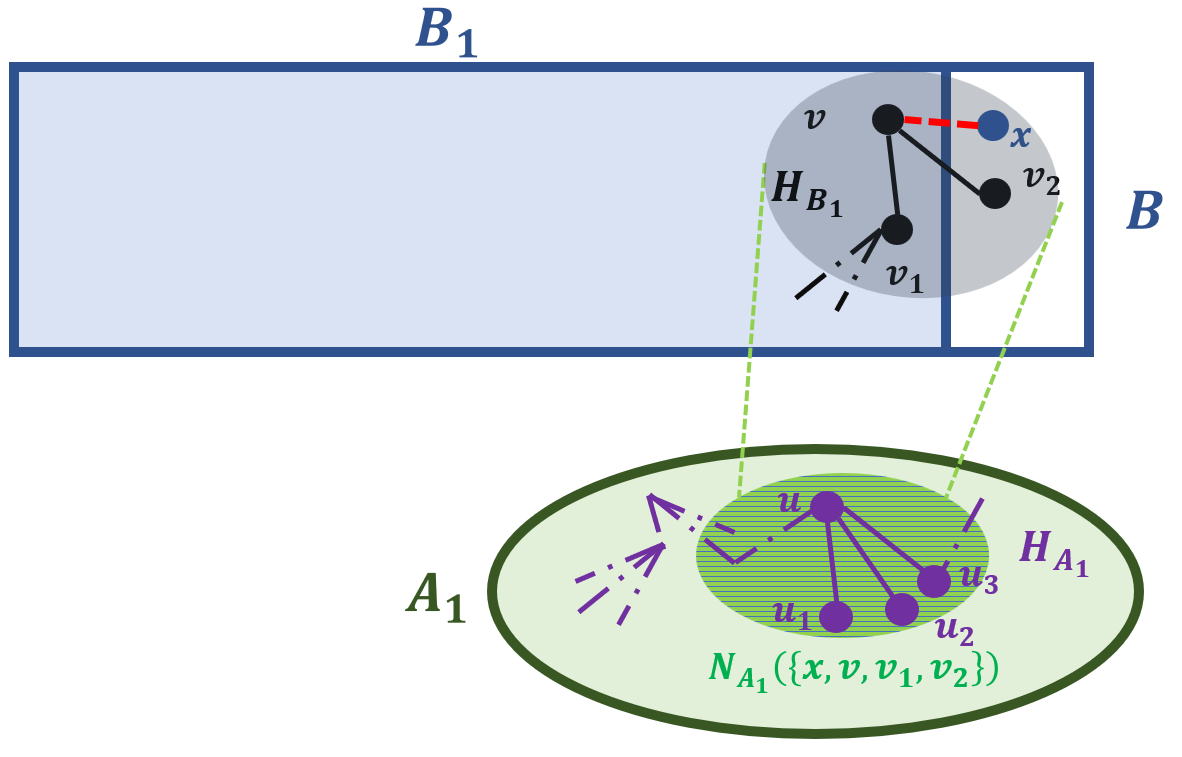}
\caption{An illustration of how the edge $\{x,v\}$, which is not induced by $B_2$, is completed by $H_{1}$. In this case, we consider $K_{2,3,3}$. In $H_{B_1}$ we have the vertices $v, v_1\in B_1$ and $v_2\in B_2$, with the edges $\{v,v_1\}$ and $\{v,v_2\}$. In the common neighbourhood in $A_1$ of $\{x,v,v_1,v_2\}$, we can find inside $H_{A_1}$ a copy of $K_{1,3}$. Note that with the edges of $H_1$ and the edge $\{x,v\}$, we now have a copy of $K_{2,3,3}$, with its parts being $\{v,u\}, \{v_1,v_2,x\}, \{u_1, u_2, u_3\}$.}
\label{f: theorem 3}
\end{figure}
\begin{claim}\label{c: zero}
    \textbf{Whp}, $H_1$ completes all but $o(n \ln n)$ of the vertex pairs in $B$ not induced by $B_2$.
\end{claim}
\begin{proof}
   Fix $S\subset A_1$, $u\in B$. We say that $u$ \textit{avoids} $S$ if $|N_G(u,S)|<\frac{|S|}{\ln^2|S|}$. The number of neighbours of $u$ in $S$ in $G$ is distributed according to $\text{Bin}(|S|, p)$. Thus, by Lemma \ref{l: chernoff}\eqref{l: chernoff 3}, the probability that $u$ avoids $S$ is at most  $(1-p)^{|S|-|S|/\ln|S|}$. In particular, if $|S|=\left(1+\frac{1}{2}\epsilon\right)\log_{\rho}n$, then the probability that $u$ avoids $S$ is at most
   \begin{align*}
       (1-p)^{|S|-|S|/\ln|S|}\le (1-p)^{\left(1+\frac{1}{3}\epsilon\right)\log_{\rho}n}= n^{-1-\frac{1}{3}\epsilon}.
   \end{align*}
   Fix a vertex $v\in B$ and expose the edges in $G$ from $v$ to $A_1$. Assume that $v$ is $A_1$-good. Fix $S\subset N_{G}(v,A_1)$ of size $\left(1+\frac{1}{2}\epsilon\right)\log_{\rho}n$. Let $X_S$ be the random variable counting the number of vertices $u\in B\setminus\{v\}$ that avoid $S$. Then $X_S$ is stochastically dominated by $\text{Bin}\left(|B|-1, n^{-1-\frac{1}{3}\epsilon}\right)$. Since $|B|-1<n$, by Lemma \ref{l: chernoff}\eqref{l: chernoff 2},
   \begin{align*}
       \mathbb{P}\left(X_S\ge \sqrt{\ln n}\right)\le \left(\frac{e}{\sqrt{\ln n}\cdot n^{\frac{1}{3}\epsilon}}\right)^{\sqrt{\ln n}}\le \exp\left(-\frac{1}{3}\epsilon \ln^{\frac{3}{2}}n\right).
   \end{align*}
   
   Note that since $v$ is $A_1$-good, then $|N_{G}(v,A_1)|\leq (1+(1+2\gamma)\epsilon) \log_{\rho} n$. Hence, by the union bound, the probability that there exists $S\subset N_{G}(v,A_1)$ of size $\left(1+\frac{1}{2}\epsilon\right)\log_{\rho}n$ such that $X_S>\sqrt{\ln n}$ is at most
     \begin{align*}
        &\binom{(1+(1+2\gamma)\epsilon) \log_{\rho} n}{\left(1+\frac{1}{2}\epsilon\right)\log_{\rho} n} \exp\left(-\frac{1}{3}\epsilon \ln^{\frac{3}{2}}n\right) \\ 
        &\le \exp\left(\frac{3}{2}\epsilon \ln\left(\frac{e(1+2\epsilon)}{\frac{3}{2}\epsilon}\right) \log_{\rho} n\right)\exp\left(-\frac{1}{3}\epsilon \ln^{\frac{3}{2}}n\right)=o\left(\frac{1}{n}\right).
    \end{align*}

    Thus, by the union bound \textbf{whp} there are no vertices $v\in B_1$ such that $X_S>\sqrt{\ln n}$ for some $S\subset N_G(v,A_1)$ of size $\left(1 + \frac{1}{2}\epsilon\right)\log_{\rho} n$.

    Now, fix an edge $\{u, v\} \subset B$ in $G$ but not in $H_1$, such that $v \in B_1$. Since $v\in B_1$, we have that $v\in V(H_{B_1})$ and has degree $s_2-1$ in $H_{B_1}$. Therefore, this edge closes a copy of $K_{1, s_2}$ with the neighbours of $v$ in $H_{B_1}$, denote them by $v_1, \dots, v_{s_2 - 1}$. Let $S\coloneqq S(v,v_1,\ldots, v_{s_2-1})$ be the set of common neighbours of $v,v_1,\ldots, v_{s_2-1}$ in $A_1$ in the graph $G$. We note that the size of $S$ must be at least $\left(1 + \frac{1}{2}\epsilon\right) \log_{\rho}n$. Indeed, we have
    \begin{align*}
        |S| &\ge |N_{G}(v, A)| - \sum_{i=1}^{s_2 - 1}|N_{G}(v, A) \setminus N_{G}(v_i, A)| \\
        &\ge |N_{G}(v, A)| - (s_2 - 1)(8 \gamma \epsilon) \log_{\rho}n
        \\&\ge (1 + \epsilon - 8s_2 \gamma \epsilon) \log_{\rho}n
        \\&\ge \left(1 + \frac{1}{2}\epsilon\right) \log_{\rho}n,        
    \end{align*}
     where the second inequality follows from the fact that, by Lemma \ref{lemma:H_B_1-subgraph}, given $\{x,y\}\in E(H_{B_1})$ we have that $d_{G}(x, y| A_1) \ge (1 + (1-6\gamma)\epsilon) \log_{\rho} n$, and the last inequality is true if $\gamma \le \frac{1}{16 s_2}$. Note that if $u$ has more than $\frac{|S|}{\ln^2 |S|}$ neighbours in $S$, then since $H_{A_1}$ is $(1/\ln^3|A_1|)$-dense with respect to $K_{s_1 - 1, s_3, \dots, s_{\ell}}$, we can find a copy of $K_{s_1 - 1, s_3, \dots, s_\ell}$ in $S$ such that the edge $\{v,u\}$, joined with the neighbours of $v$ in $B_1$, completes a copy of $K_{s_1, \dots, s_\ell}$. Thus, by this and the above, \textbf{whp} there are at most $n\sqrt{\ln n}=o(n \ln n)$ non-completed edges from $B$ not induced by $B_2$.    
\end{proof}
    
We now turn our attention to the edges induced by $B_2$. By Lemma \ref{l: ABlemma}, we can take a spanning subgraph $H_{A_2} \subseteq G[A_2]$ which is $(1/\ln|A_2|)$-dense with respect to $K_{s_1 - 1, s_3, \dots, s_\ell}$ and $\{K_\ell, K_{s_1}^{(\ell-1)}\}$-free. Furthermore, let $H_{B_2} \subset G[B_2]$ be a $K_{s_2}$-factor (that is, a spanning subgraph composed of vertex disjoint $K_{s_2}$ --- see \cite{JKV08}) of $G[B_2]$, which is in particular $K_{1, s_2}$-saturated and $K_{s_1, s_2 - s_1 + 1}$-free. Let $H_2$ be the graph on $A_2 \cup B_2$ with the edges of $H_{A_2}$, $H_{B_2}$, and all the edges between $A_2$ and $B_2$ in $G$.
    
Note that we can choose $L$ large enough such that every $s_2 + 1$ vertices from $B_2$ have at least $\frac{a_2p^{s_2+1}}{10}$ neighbours in $A_2$ in $G$. Thus, when we add an edge induced by $B_2$, we close a copy of $K_{1, s_2}$ which has a large neighbourhood in $A_2$, in which we can find a copy of $K_{s_1 - 1, s_3, \dots, s_\ell}$ and form a copy of $K_{s_1,s_2,\ldots, s_{\ell}}$.
    
Observe that \textbf{whp} we have $\Theta(n \log n)$ edges between $A_2$ and $B_1$ in $G$. We now utilise the set $A_3$ to complete them. Indeed, by Lemma \ref{l: chromatic number} \textbf{whp} $k \coloneqq \chi(G[A_2]) = O\left(\frac{a_2}{\ln{a_2}}\right)$. We can then split $A_2$ to $k$ colour classes $A_2^1, \dots, A_2^k$. Thus, there are no edges of $G$ (and thus of $H_{A_2}$) inside $A_2^i$, for every $i \in [k]$. We further partition the vertices of $A_3$ to $2k$ (almost) equal parts $A_3^1, \dots, A_3^{2k}$ of size $a_4 \coloneqq \frac{a_3}{2k} = \Theta(\sqrt{\ln a_2})$. For every $i \in [2k]$, by Lemma \ref{l: ABlemma}, the probability that there exists a subgraph $H_{A_3^i} \subset G[A_3^i]$ which is $(1/\ln\ln\ln^4 n)$-dense with respect to $K_{s_1 - 1, s_3, \dots, s_\ell}$ and $\{K_\ell, K_{s_1}^{(\ell-1)}\}$-free is $1 - o(1)$. Hence, by standard binomial tail bounds, \textbf{whp} there exist distinct $i_1, \dots, i_k$ such that there exists such a subgraph in each $A_3^{i_j}$, for every $j \in [k]$. Therefore, for every $j \in [k]$, there is $H_{A_3^{i_j}} \subseteq G[A_3^{i_j}]$ which is $(1/\ln\ln\ln^4 n)$-dense with respect to $K_{s_1 - 1, s_3, \dots, s_\ell}$ and $\{K_\ell, K_{s_1}^{(\ell-1)}\}$-free. We let $A_3^j=A_3^{i_j}$ and $H_{A_3^j}=H_{A_3^{i_j}}$ for simplifying notations and without loss of generality.
    
Let $H_3$ be the graph with edges of $H_{A_2}$, $H_{A_3^1}, \dots, H_{A_3^k}$, together with the edges between $A_3$ and $B_1$ in $G$, and together with the edges between $A_3^i$ to $A_2^i$, for every $i \in [k]$, in $G(n,p)$. Set $H \coloneqq H_1 \cup H_2 \cup H_3$. 
    
\begin{claim}\label{c: one}
    \textbf{Whp}, $H$ completes all but at most $o(n \log n)$ edges from $B_1$ to $A_2$.
\end{claim}
\begin{proof}
    Fix two vertices $v \in B_1$ and $u \in A_2^i$, for some $i \in [k]$.
        
    Recall that $v$ has $s_2 - 1$ neighbours in $B_1$. Thus, together with $u$, we close a copy of $K_{1,s_2}$. If this copy has more than $\frac{a_4}{\log^3 a_4}$ common neighbours in $A_3^i$, then we can find among these neighbours a copy of $K_{s_1 - 1, s_3, \dots, s_\ell}$ which closes a copy of $K_{s_1, \dots, s_\ell}$. By Lemma \ref{l: chernoff}\eqref{l: chernoff 3}, the probability that this copy of $K_{1, s_2}$ has less than $\frac{a_4}{\log^3 a_4}$ common neighbours is at most
    \begin{align*}
        (1 - p^{s_2 + 1})^{a_4 - a_4 / \ln a_4} \le (1 - p^{s_2 + 1})^{\Omega(\sqrt{\ln a_2})}.
    \end{align*}
    Hence, the expected number of uncompleted edges $e=\{v,u\}$ with $v \in B_1$ and $u \in A_2$ is at most
    \begin{align*}
        (1 - p^{s_1 + 1})^{\Omega(\sqrt{\ln a_2})} \cdot |B_1| \cdot |A_2| = (1 - p^{s_1 + 1})^{\Omega(\sqrt{\ln a_2})} \cdot O(n \log_{\rho}n).
    \end{align*}   
    Therefore, by Markov's inequality, \textbf{whp} the number of uncompleted such pairs is $o(n \log_{\rho}n)$.
\end{proof}
    
\begin{claim}\label{c: two}
    \textbf{Whp}, $e(H) = (1+\Theta(\epsilon)) n \log_{\rho}n$.
\end{claim}
\begin{proof}
    Indeed, \textbf{whp}
    \begin{align*}
        e(H) &\le e(H_1) + e(H_2) + e(H_3) \\
        &\le (1+o(1))a_1 np + e(H_{B_1}) + e(H_{A_1}) + a_2 |B_2| p + e(H_{A_2}) + e(H_{A_3}) + a_3 np \\
        &\le (1+o(1))(1+(s_1+1)\epsilon) n \log_{\rho}n= (1+\Theta(\epsilon))n\log_{\rho}n.
    \end{align*}
\end{proof}
    
\begin{claim}\label{c: three}
    $H$ is $K_{s_1, \dots, s_\ell}$-free.
\end{claim}
\begin{proof}
    Suppose towards contradiction that there exists a copy of $K_{s_1, \dots, s_\ell}$ in $H$, denote this copy by $F$.
        
    Assume first that $V(F) \cap A_3 \neq \varnothing$. Set $\Tilde{F} \coloneqq V(F) \cap A_3$. Since $H[A_3]$ is $K_\ell$-free, then at least one full part of $K_{s_1, \dots, s_\ell}$ must come from $B_1 \cup A_2$, as all its vertices must be adjacent to every vertex of $\Tilde{F}$. If only one full part comes from $B_1 \cup A_2$, then $H[A_3]$ must contain a copy of $K_{s_1}^{(\ell-1)}$ --- a contradiction, as by construction $H[A_3]$ is composed of vertex disjoint subsets, each of which is $K_{s_1}^{(\ell-1)}$ free. Thus, $B_1\cup A_2$ contains at least one full part as well as an additional vertex from another part of $K_{s_1,\dots, s_{\ell}}$. As this additional vertex is adjacent to the full part lying in $B_1\cup A_2$ and there are no edges between $B_1$ and $A_2$, all these vertices must belong exclusively to $B_1$ or exclusively to $A_2$. If they belong to $B_1$, then since the maximum degree of $H_{B_1}$ is $s_2-1$ we have that the full part must be of size $s_1$. Furthermore, we have that $H_{B_1}$ is $K_{s_1,s_2-s_1+1}$-free. Thus, if they belong to $B_1$, $H[A_3]$ must contain $K_{s_1}^{(\ell-1)}$ in order to complete $K_{s_1,\dots, s_{\ell}}$, as there are $\ell-1$ parts missing at least $s_1$ vertices, once again leading to contradiction. Finally, if they belong to $A_2$, then they should be split between different independent sets $A_2^i$ and $A_2^j$, but there are no vertices in $A_3$ which are adjacent to both $A_2^i$ and $A_2^j$ in $H$ --- a contradiction.
    
    Let us now assume that $V(F) \cap A_3 = \varnothing$ and $V(F) \cap A_2 \neq \varnothing$. Set $\Tilde{F} \coloneqq V(F) \cap A_2$. Since $H[A_2]$ is $K_\ell$-free, then a full part of $K_{s_1, \dots, s_\ell}$ must come from $B_2$ and, as before, at least one additional vertex. This vertex has neighbours in the full part, and thus, since $H[B_2]$ is a $K_{s_2}$-factor, the full part must be of size $s_1$. Moreover, there are at most $s_2-s_1$ vertices of $F$ in $B_2$ that do not belong to this part of size $s_1$. Therefore, each part of $\Tilde{F}$ must contain at least $s_1$ vertices, creating a copy of $K_{s_1}^{(\ell-1)}$ --- a contradiction to the fact that $H[A_2]$ is $K_{s_1}^{(\ell-1)}$-free.  Similarly, if $V(F)\cap A_3=V(F)\cap A_2=\varnothing$ and $V(F)\cap A_1\neq \varnothing$, since $H[A_1]$ is $K_{\ell}$-free, we once again obtain a contradiction. Thus, we may assume that $V(F)\subseteq B$. However, 
    since there are no edges between $B_1$ and $B_2$ in $H$, we must either $V(F)\subseteq B_1$ or $V(F)\subseteq B_2$, contradicting the fact that $H[B_1]$ and $H[B_2]$ are $K_{s_1,\ldots,s_{\ell}}$-free.
    \end{proof}

In conclusion, \textbf{whp} by Claim \ref{c: three} $H$ is $K_{s_1,\ldots, s_{\ell}}$-free; by Claims \ref{c: zero} and \ref{c: one} \textbf{whp} $H$ completes all but at most $o(n\ln n)$ of the edges, and adding each of these edges that does not close a copy of $F$ until none remains results in a $K_{s_1,\ldots, s_{\ell}}$-saturated graph; finally, by Claim \ref{c: two} and by the additional edges mentioned before, \textbf{whp} $e(H)=\left(1+\Theta(\epsilon)\right)n\log_{\rho}n+o(n\ln n)=\left(1+\Theta(\epsilon)\right)n\log_{\rho}n$. As we may choose $\epsilon$ arbitrarily small, this completes the proof of Theorem~\ref{th: complete mp}.

\subsection{Proof of Lemma \ref{lemma:H_B_1-subgraph}}
We build such a subgraph iteratively. In each iteration, we find a large matching in $G[B_1]$, which we add to the subgraph, such that the matching satisfies the following:
\begin{enumerate}
    \item The union of the previous matching together with this matching does not induce a copy of $K_{s_1,s_2-s_1+1}$; and,
    \item the endpoints of every edge in the matching have a large common neighbourhood in $A_1$ in $G$.
\end{enumerate}

Denote by $\Gamma$ the auxiliary graph with vertex set $B_1$ and the set of edges defined as follows. For every two vertices $v\neq u$ in $B_1$, 
\begin{align*}
    \{u, v\} \in E(\Gamma) \Longleftrightarrow d_{G}(u, v| A_1) \ge (1 + (1-6\gamma)\epsilon) \log_{\rho} n.
\end{align*}
Set $p' = 1 - (1-p)^{1/(s_2 - 1)}$. For every $i \in [s_2 - 1]$, denote by $\Gamma_i$ the subgraph of $\Gamma$ obtained by retaining every edge independently with probability $p'$. We have $\Gamma_p\coloneqq \Gamma \cap G(n,p)$. Note that $\Gamma_p$ has the same distribution as $\bigcup_{i=1}^{s_2 - 1} \Gamma_i$.

We will in fact prove the following equivalent lemma:
\begin{lemma}
    \textbf{Whp} there exists $H_{B_1} \subseteq \Gamma_p$ which is $K_{s_1, s_2 - s_1 + 1}$-free, has maximum degree $s_2-1$, and all but $O\left(\frac{n}{\log n}\right)$ of its vertices have degree $s_2-1$.
\end{lemma}

We consider two cases separately. In the first case, we assume that $p=\frac{1}{2}$. While some details will be different when $p>\frac{1}{2}$, we believe the key ideas, in particular the ball-covering technique (see Claim \ref{claim:hamming-balls}) are clearer in this case. Afterwards, we mention how to complete the proof for the range of $p>\frac{1}{2}$, where, in particular, Claim \ref{claim:phi-1-to-1} no longer necessarily holds.

\subsubsection[Proof of Lemma \ref{lemma:H_B_1-subgraph}, p=0.5]{Proof of Lemma \ref{lemma:H_B_1-subgraph}, $p=0.5$}
    Recall that we are seeking a subgraph of $\Gamma_p$ which is $K_{s_1,s_2-s_1+1}$-free, has maximum degree $s_2-1$ and all but $O\left(\frac{n}{\log n}\right)$ of its vertices are of degree $s_2-1$. 

    As for the first requirement, note that a graph whose maximum degree is $s_2-1$ and is $C_4$-free, is $K_{s_1,s_2-s_1+1}$-free graph. Indeed, if $s_1=s_2$ or $s_1=1$, we have that $K_{s_1,s_2-s_1+1}=K_{1,s_2}$, and thus asking for the maximum degree to be $s_2-1$ suffices. Otherwise, $1<s_1< s_2$ and we have that any copy of $K_{s_1,s_2-s_1+1}$ contains $K_{2,2}$, that is, $C_4$.
    
    As for the second requirement, our strategy will then be to find a sufficiently large matching $M_i$ in $\Gamma_i$, for every $i \in [s_2 - 1]$, such that there are no copies of $C_4$ in $M:=\bigcup_{i=1}^{s_2 - 1} M_i$. We will show that for every $i\in [s_2-1]$, there are at most $O\left(\frac{n}{\log n}\right)$ vertices that are unmatched. Thus, the subgraph whose edges are the edges of $M$ would then be the desired subgraph of $\Gamma_p$.
    
     Fix $i \in \left\{2, \ldots s_2 - 1\right\}$. Assume that there exist edge-disjoint matchings $M_1 \subseteq \Gamma_1, \dots, M_{i-1} \subseteq \Gamma_{i-1}$ such that $\bigcup_{j=1}^{i - 1} M_j$ is $C_4$-free. We will find a matching $M_i \subseteq \Gamma_i$ such that $\bigcup_{j=1}^{i} M_j$ is $C_4$-free. 

     Recall that $I = \Big[\left(1 + \epsilon\right) \log_{\rho} n,\ \left(1 + (1+2\gamma)\epsilon\right) \log_{\rho} n\Big]$. Set
    \[
        W = \{x \subseteq A_1 \colon |x| \in I\}.
    \]    
    Denote by $G_W$ the graph with vertex set $W$ and the set of edges defined as follows. For every $x \neq y \in W$, 
    \begin{align}\label{align:G_W-edges-def}
        \{x,y\} \in E(G_W) \Longleftrightarrow |x \cap y| \ge (1+(1-6\gamma)\epsilon) \log_{\rho} n.
    \end{align}
    Define $\phi:V(\Gamma)\to V(G_W)$ such that $\phi(v)=N_{G}(v| A_1)$ for every $v \in V(\Gamma)$. Note that this definition is valid since if $v \in V(\Gamma)$, then $v$ is $A_1$-good and thus $N_{G}(v| A_1) \in W$.

    \begin{claim}\label{claim:phi-1-to-1}
        \textbf{Whp} $\phi$ is injective.
    \end{claim}
    \begin{proof}
        Fix $u \in B$. If $u\in V(\Gamma)$, then for every vertex $v \in V(\Gamma)$, we have $\phi(u) = \phi(v)$ if and only if $N_{G}(v| A_1) = N_{G}(u| A_1)$. For every vertex $v \in B$,
        \begin{align*}
            \mathbb{P}\left(N_{G}(v| A_1) = N_{G}(u| A_1)\right) &= p^{d_{G}(u| A_1)}(1-p)^{a_1 - d_{G}(u| A_1)}=\left(\frac{1}{2}\right)^{a_1} \\
            &= \left(\frac{1}{2}\right)^{\frac{1}{p}\left(1 + (1+\gamma)\epsilon\right) \log_{2} n} =o(1/n^2).
        \end{align*}
        Hence, by the union bound, \textbf{whp} there are no two vertices $u \neq v\in V(\Gamma)$ such that $\phi(u) = \phi(v)$.
    \end{proof}

    Set $\widetilde{G}_W = G_W\left[\phi\left(V(\Gamma)\right)\right]$. Denote by $\widetilde{G}_W(p')$ the random subgraph of $\widetilde{G}_W$ obtained by retaining every edge of $\widetilde{G}_W$ independently with probability $p'$. By Claim \ref{claim:phi-1-to-1}, \textbf{whp} $\phi$ is injective. Therefore, $\widetilde{G}_W \cong \Gamma$. Recall that we want to find a matching in $\Gamma_i$. We will show that \textbf{whp} $\alpha(\Gamma_i) \le \frac{n}{\log_{\rho}n}$ and later we will construct the desired matching. Since $\widetilde{G}_W \cong \Gamma$, it suffices to prove the following lemma.    
    
    \begin{lemma}\label{claim:bound-on-alpha-G_W}
        \textbf{Whp}
        \[
            \alpha(\widetilde{G}_W(p')) \le \frac{n}{\log_{\rho}n}.
        \]
    \end{lemma}
    Before proving this lemma, let us outline how we shall use the notion of ball-covering in the proof. Let us recall that $x,y$ are not connected in $\widetilde{G}_W$ if $|x\cap y|< \left(1+(1-6\gamma)\epsilon\right)\log_{\rho}n$. Consider a Hamming ball around $x\in W$, containing all $y\in W$ such that $|x\cap y|$ is sufficiently large. If we can find $m$ vertices, such that the Hamming balls around them cover all the vertices of $W$ (and the respective edges are retained in $\widetilde{G}_W(p')$), then the independence number of the graph is at most $m$ --- indeed, any set of more than $m$ vertices must have two vertices in the same Hamming ball, and thus there must be an edge between them. Let us proceed with the detailed proof.
    \begin{proof}
        For every $x \in W$, denote by $B(x)$ the following Hamming ball around $x$:
        \[
            B(x) \coloneqq \{y \in W \colon |x \cap y| \ge (1 + (1-\gamma)\epsilon)\log_{\rho} n\}.
        \]
        Set $m \coloneqq \frac{n}{\log(n)}$ and $m' \coloneqq \frac{n}{\log^3 n}$. Let $B' = \{y_1, \dots, y_{m'} \} \subseteq V(\Gamma)$ be an arbitrary subset of size $m'$.
        \begin{claim}\label{claim:hamming-balls}
            \textbf{Whp}, for every $x \in W$, there exists a vertex $v \in B'$ such that 
            \[
                d_{G}(v| x) \ge (1 + (1-\gamma)\epsilon)\log_{\rho} n.
            \]
        \end{claim}
        \begin{proof}
            For every $x \in W$, set
            \[
                B'(x) = \{ v \in B' \colon d_{G}(v| x) \ge (1 + (1-\gamma)\epsilon)\log_{\rho} n \}.
            \]
                
            Fix $x \in W$. Since $|B'| \ge \frac{n}{\log^3 n}$, by Lemma \ref{l: many neighbours} the probability that $B'(x) = \varnothing$ is at most $\exp\left(-n^{0.5\epsilon}\right)$. Note that $|W| = n^{O(1)}$. Thus, by the union bound over $W$, the probability that there exists a vertex $x \in W$ such that $B'(x) = \varnothing$ is at most
            \[
                n^{O(1)} \cdot \exp\left(-n^{0.5\epsilon}\right) = o(1).\qedhere
            \]
        \end{proof}
       Let $\mathcal{A}$ be the event that for every $x \in W$, there exists a vertex $v \in B'$ such that $d_{G}(v| x) \ge (1 + (1-\gamma)\epsilon)\log_{\rho} n$. By Claim \ref{claim:hamming-balls}, we have that $\mathbb{P}(\mathcal{A})=1-o(1)$. In the following claim, we assume that $\mathcal{A}$ holds deterministically.
        \begin{claim}\label{c: number of edges}
            For every $J = \{x_1, \dots, x_m\} \subset V(\widetilde{G}_W)$, we have  $|E(\widetilde{G}_W[J])|\ge 0.25 n \log_{\rho}n$.
        \end{claim}
        \begin{proof}
            For every $i \in [m']$, set
            \[
                Y_i \coloneqq \{x \in J \colon N_{G}(y_i| A_1) \in B(x)\}.
            \]
            Note that by $\mathcal{A}$, we have that $\bigcup_{i=1}^{m'}Y_i=J$. Indeed, for every $x \in J$, there exists a vertex $y \in B'$ such that $N_{G}(y| A_1) \in B(x)$ and thus there exists $i \in [m']$ such that $x \in Y_i$.
            
            We now show that $\widetilde{G}_W[Y_i]$ is a clique for every $i \in [m']$. Fix $i \in [m']$ and two different vertices $x, x' \in Y_i$. By the definition of $Y_i$,
                \[
                    |x \cap N_{G}(y_i| A_1)| = |N_{G}(y_i| x)| \ge (1 + (1-\gamma)\epsilon)\log_{\rho} n
                \]
                and
                \[
                    |x' \cap N_{G}(y_i| A_1)| = |N_{G}(y_i| x')| \ge (1 + (1-\gamma)\epsilon)\log_{\rho} n.
                \]
                Recall that $y_i \in B_1$, so $y_i$ is $A_1$-good and thus $d_{G}(y_i| A_1) \le (1 + (1+2\gamma)\epsilon)\log_{\rho} n$. Then,
                \[
                    |N_{G}(y_i| A_1) \setminus x| \le (3\gamma \epsilon) \log_{\rho} n \quad \text{and} \quad |N_{G}(y_i| A_1) \setminus x'| \le (3\gamma \epsilon) \log_{\rho} n.
                \]
                Hence,
                \begin{align*}
                    |x \cap x'| &\ge |x \cap x' \cap N_{G}(y_i| A_1)| \\&\ge |N_{G}(y_i| A_1)| - |N_{G}(y_i| A_1) \setminus x| - |N_{G}(y_i| A_1) \setminus x'| \\& \ge (1 + \epsilon - 6 \gamma \epsilon)\log_{\rho} n,
                \end{align*}                
                where the last inequality is true since $d_{G}(y_i| A_1) \ge (1 + \epsilon) \log_{\rho} n$ because $y_i$ is $A_1$-good. Therefore, by \eqref{align:G_W-edges-def}, $\{x, x'\} \in E(\widetilde{G}_W)$ implying that $\widetilde{G}_W[Y_i]$ is indeed a clique.

            For every $i \in [m']$, set
            \begin{align*}
                Y_i' &\coloneqq Y_i \setminus \left(\bigcup_{j=1}^{i - 1} Y_j\right).
            \end{align*}

            Since for every $i \in [m']$ we have that $\widetilde{G}_W[Y_i]$ is a clique, $\widetilde{G}_W[Y_i'] \subseteq \widetilde{G}_W[Y_i]$ is also a clique. Thus,
            \begin{align*}
                |E(\widetilde{G}_W[J])| &\ge \sum_{i=1}^{m'} \binom{|Y_i'|}{2}
                \ge m' \cdot \binom{\frac{m}{m'}}{2} \ge m' \cdot \left(\frac{m}{2m'}\right)^2 \\
                &= \frac{m^2}{4m'} = \frac{1}{4} n \log_{\rho}n,
            \end{align*}
            where the second inequality is true by Jensen's inequality and the fact that $m = |J| = \sum_{i=1}^{m'} |Y_i'|$. 
        \end{proof}
        By Claim \ref{c: number of edges}, 
            \[
                \mathbb{P}_{p'}(|E(\widetilde{G}_W(p')[J])| = 0 \mid \mathcal{A}) = (1-p')^{|E(\widetilde{G}_W[J])|} \le (1-p')^{0.25 n \log_{\rho}n}. 
            \]
        We have, 
        \begin{align*}
            \binom{|W|}{m} \le |W|^m = \exp\left(\Theta(m \log n)\right) = \exp\left(\Theta(n)\right).
        \end{align*}
        Therefore, by the union bound, the probability that there is an independent set in $\widetilde{G}_W(p')$ of size $m$ is at most
        \[
            \mathbb{P}\left(\neg A\right)+\binom{|W|}{m} (1-p')^{0.25 n \log_{\rho}n} = o(1). \qedhere
        \]
        
    \end{proof}
    Recall that our goal is to find a matching $M_i \subseteq \Gamma_i \setminus (\cup_{j=1}^{i-1} M_j)$ such that $\cup_{j=1}^{i} M_j$ is $C_4$-free.
    By Claim \ref{claim:bound-on-alpha-G_W} and the fact that $\widetilde{G}_W \cong \Gamma$, \textbf{whp}
    \[
        \alpha(\Gamma_i) \le \frac{n}{\log_{\rho}n}.
    \]
    Let $M_i \subseteq \Gamma_i \setminus (\cup_{j=1}^{i-1} M_j)$ be a matching of the maximum size such that $\bigcup_{j=1}^{i} M_j$ is $C_4$-free.

    Let $U \subseteq V(\Gamma_i)$ be the set of unmatched vertices. The next claim bounds the maximum degree $\Delta(\Gamma_i[U])$.
    \begin{claim}\label{claim:I-maximal-degree}
        $\Delta(\Gamma_i[U]) \le {s_2}^3$.
    \end{claim}
    \begin{proof}            
        Suppose towards contradiction that $\Delta(\Gamma_i[U]) > {s_2}^3$ and take a vertex $v \in U$ such that $d_{\Gamma_i[U]}(v) > {s_2}^3$.

        Note that there are at most ${s_2}^3$ many paths with three edges in $\bigcup_{j=1}^{i} M_j$ which start with the vertex $v$. Thus, there exists a vertex $u \in N_{\Gamma_i[U]}(v)$ such that $u$ is not adjacent to another endpoint of the above paths and thus $\{u,v\}$ does not close a copy of $C_4$ in $\bigcup_{j=1}^{i} M_j$. Hence, we can add $\{u,v\}$ to the matching $M_i$, a contradiction to the maximality of $M_i$.
    \end{proof}
    
    We finish with the following claim.
    \begin{claim}\label{c: small degree}
        \textbf{Whp} $|U| \le (s_2^3 + 1) \frac{n}{\log_{\rho}n}$.
    \end{claim}
    \begin{proof}
        Suppose towards contradiction that $|U| > (s_2^3 + 1) \frac{n}{\log_{\rho}n}$. By Claim \ref{claim:I-maximal-degree}, $\Delta(\Gamma_i[U]) \le {s_2}^3$. Hence,
        \[
            \alpha\left(\Gamma_i[U]\right) \ge \frac{|U|}{\Delta(\Gamma_i[U]) + 1} > \frac{n}{\log_{\rho}n},
        \]
        a contradiction to Lemma \ref{claim:bound-on-alpha-G_W}.
    \end{proof}  
    The desired subgraph in Lemma \ref{lemma:H_B_1-subgraph} is the subgraph $H_{B_1}$ with its edges from $\cup_{i=1}^{s_2 - 1} M_i$ --- indeed, note that by Claim \ref{c: small degree}, there are at most $O\left(\frac{n}{\log n}\right)$ unmatched vertices at each round, and thus at most $O\left(\frac{n}{\log n}\right)$ vertices of degree at most $s_2-1$.

\subsubsection[Proof of Lemma \ref{lemma:H_B_1-subgraph}, p>0.5]{Proof of Lemma \ref{lemma:H_B_1-subgraph}, $p>0.5$}\label{s: problematic section}
Let us now explain how to complete the proof for $p > 0.5$. Indeed, note that Claim \ref{claim:phi-1-to-1} is not necessarily true, and hence we require a more delicate treatment to overcome the lack of isomorphism. Recall that our main goal is to show that $\alpha(\Gamma_i) \le \frac{n}{\log_{\rho}n}$. We prove this with the following series of relatively short claims and lemmas, where the key idea is that one can consider equivalence classes under $\phi$, and show that \textbf{whp} either all such classes are of bounded order, or of polynomial order (Claim \ref{l: key p>0.5}).

For every $v, u \in V(\Gamma)$, we say that $v \sim u$ if and only if $\phi(v) = \phi(u)$. Let $C_1, \dots, C_\ell$ be the equivalence classes under this relation. With a slight abuse of notation, given $C_{j}=\{v_1,\ldots,v_k\}$ we write $\phi(C_j)\coloneqq \phi(v_1)$.

\begin{claim}\label{l: key p>0.5}
    \textbf{Whp} one of the following holds.
    \begin{enumerate}
        \item There exists a constant $C$ such that
        \[
            |C_j| \le C,\quad \forall j \in [\ell].
        \]

        \item There exists $\beta > 0$ such that
        \[
            |C_j| \ge n^\beta,\quad \forall j \in [\ell].
        \]
    \end{enumerate}
\end{claim}
\begin{proof}
    For every $v \in [n] \setminus A_1$, set $X_v = 0$ if $v$ is  $A_1$-bad and $X_v$ to be the number of vertices $u \sim v$ otherwise. Fix $v \in [n] \setminus A_1$ and consider the random variable $\epsilon' = \epsilon'(v)$ such that $|N_{G(n,p)}(v) \cap A_1| = (1 + \epsilon') \log_{\rho} n$. Note that if $v$ is $A_1$-good, then $\epsilon' \in [\epsilon, (1+2\gamma)\epsilon]$.  We have
    \begin{align*}
        \mathbb{E}[X_v | N_{G(n,p)}(v) \cap &A_1\ \text{and $v$ is $A_1$-good}] = (1 + o(1)) n p^{|N_{G(n,p)}(v) \cap A_1|} (1-p)^{|A_1| - |N_{G(n,p)}(v) \cap A_1|}
        \\&= (1 + o(1)) n p^{(1 + \epsilon')\log_{\rho} n} (1-p)^{\frac{1 + (1+\gamma)\epsilon}{p} \log_{\rho} n - (1 + \epsilon')\log_{\rho} n}
        \\&= (1 + o(1)) n (1-p)^{(\log_{1-p} p) (1 + \epsilon')\log_{\rho} n} (1-p)^{\frac{1 + (1+\gamma)\epsilon}{p} \log_{\rho} n - (1 + \epsilon')\log_{\rho} n}
        \\&= (1 + o(1)) n^{1 - (1 + \epsilon') (\log_{1-p} p) - \frac{1 + (1+\gamma)\epsilon}{p} + 1 + \epsilon'}
        \\&= (1 + o(1)) n^{2 - \log_{1-p}p - \frac{1}{p} + \epsilon' - \epsilon' \log_{1-p}p - \frac{(1+\gamma)\epsilon}{p}}.
    \end{align*}


    Set $f(x) \coloneqq 2 - \log_{1-x} (x) - \frac{1}{x}$. Let $x'$ be in $(0,1)$ such that $f(x')=0$, noting that $f(x)$ is increasing in $x \in (0,1)$ and $f(x) = 0$ around $x \approx 0.64$. We then have that for $\epsilon$ small enough and for some constant $c>0$, $\epsilon' - \epsilon' \log_{1-x'}x' - \frac{(1+\gamma)\epsilon}{x'}<-c\epsilon$. Thus, if $p \le x'$, we have that $\mathbb{E}[X_v\,|\, \text{$v$ is $A_1$-good}] \le n^{-c\epsilon}$ for every $v \in [n] \setminus A_1$, where we stress that $\epsilon$ can depend on $p$. By Lemma \ref{l: chernoff}\eqref{l: chernoff 2},
    \[
        \mathbb{P}(X_v > C\,|\, \text{$v$ is $A_1$-good}) \le \left(\frac{e}{\frac{C}{\mathbb{E}[X_v]}}\right)^{C} \le n^{-C\cdot c\epsilon} = o(1/n).
    \]
    Hence, by the union bound over all $v \in [n]\setminus A_1$, \textbf{whp} $X_v \le C$ for every $v \in V(\Gamma)$, and thus the first item of the claim holds.
    
    If $p > x'$, then we may choose $\epsilon$ small enough such that $\mathbb{E}[X_v \mid \text{$v$ is $A_1$-good}] \ge n^{\beta}$, for some $\beta > 0$, for every $v \in [n] \setminus A_1$. By Lemma \ref{l: chernoff}\eqref{l: chernoff 1},
    \[
        \mathbb{P}(X_v \le 0.5 \mathbb{E}\left[X_v \mid \text{$v$ is $A_1$-good}\right] \mid \text{$v$ is $A_1$-good}) \le \exp\left(-\Theta(n^{\beta})\right).
    \]
    Hence, by the union bound over all $v \in V(\Gamma)$, \textbf{whp} $X_v\ge 0.5 n^{\beta}$ for every $v \in [n]\setminus A_1$ which is $A_1$-good, and thus the second item of the claim holds as well.
\end{proof}
We complete the proof with the two following lemmas.
\begin{lemma}
    If there exists a constant $C$ such that
    \[
        |C_j| \le C,\quad \forall j \in [\ell],
    \]
    then \textbf{whp} $\alpha(\Gamma_i) \le \frac{n}{\log_{\rho}n}$.
\end{lemma}
\begin{proof}
    Set $\widetilde{G}_W$. Consider the following coupling $(\widetilde{G}_W(p'), \Gamma_{p'})$. $\Gamma_{p'}$ is obtained by retaining each edge independently in $\Gamma$ with probability $p'$. For every $v, u \in V(\widetilde{G}_W)$, there is an edge $\{u,v\}$ in $\widetilde{G}_W(p')$ if and only if in $\Gamma_{p'}$ there are all the edges between $\phi^{-1}(v)$ and $\phi^{-1}(u)$. Note that
    \[
        \frac{\alpha(\Gamma_{p'})}{C} \le \alpha(\widetilde{G}_W(p')).
    \]
    Indeed, let $I$ be a maximum independent set in $\Gamma_{p'}$. Assume that $I$ has vertices from $m$ different equivalence classes. Observe that $m\ge \frac{|I|}{C}$ as otherwise, by the pigeonhole principle, we have an equivalence class larger than $C$. Notice that $\phi(I)$ is also an independent set in $\widetilde{G}_W(p')$ since, for every $v, u \in I$, there is at least one edge missing in $\Gamma_{p'}$ between the equivalence classes of $v$ and $u$, and thus there is no edge between $\phi(v)$ and $\phi(u)$ in $\widetilde{G}_W(p')$. Hence, $\phi(I) \le \alpha(\widetilde{G}_W(p'))$ and thus, 
    \[
        \frac{\alpha(\Gamma_{p'})}{C} \le m=|\phi(I)|\le \alpha(\widetilde{G}_W(p')).
    \]

    By the coupling above, every edge in $\widetilde{G}_W(p')$ appears with probability at least $(p')^{C^2}$. Thus, it suffices to show that \textbf{whp} the independence number of the binomial random subgraph of $\widetilde{G}_{W}$ obtained by retaining every edge with probability $(p')^{C^2}$ is at most $\frac{n}{\log_{\rho}n}$.
    The rest of the proof is identical to the proof of Claim \ref{claim:bound-on-alpha-G_W}.
\end{proof}
\begin{lemma}
    If there exists $\beta > 0$ such that
    \[
        |C_j| \ge n^{\beta},\quad \forall j \in [\ell],
    \]
    then \textbf{whp} $\alpha(\Gamma_i) \le \frac{n}{\log_{\rho}n}$.
\end{lemma}
\begin{proof}
    Let $I$ be a maximum independent set in $\Gamma_i$. For every $j \in [\ell]$, set $I_j \coloneqq I \cap C_j$.
    
    Notice that, for every $j \in [\ell]$, $C_j$ is a clique in $\Gamma$. The probability that there is an independent set in $G(n,p)$ of size $n^{\beta} / \log_{\rho}n$ is at most
    \[
        \binom{n}{\frac{n^{\beta}}{\log_{\rho}n}} (1-p)^{\binom{n^{\beta} / \log_{\rho}n}{2}} \le \exp\left(2n^{\beta} -  \frac{pn^{2\beta}}{2\ln^2 n}\right).
    \]
    There are at most $n^{1-\beta}$ many $C_j$s, and thus, by the union bound, \textbf{whp} $I_j \le \frac{n^{\beta}}{\log_{\rho}n}$, for every $j \in [\ell]$. Hence,
    \[
        |I| \le n^{1-\beta} \cdot \frac{n^\beta}{\log_{\rho}n} = \frac{n}{\log_{\rho}n}.
    \]
\end{proof}

\subsection{Theorem \ref{th: property star}}
Let us first recall that since $F$ satisfies property $(\star)$, we may find an edge $\{u,v\}$ such that for every independent set $I\subseteq V(F)$, $F[V\setminus I]$ is non-$F[V\setminus\{u,v\}]$-degenerate. Let $\mathcal{I}\coloneqq \{I\subseteq V\colon I \text{ is an independent set of } F\}$, and let $\mathcal{F}\coloneqq\{F[V\setminus I]\colon I\in \mathcal{I}\}$. By Lemma \ref{l: ABlemma}, \textbf{whp} there exists a graph which is $\mathcal{F}$-free, and is $n^{-\delta}$-dense with respect to $F[V\setminus\{u,v\}]$. 

The proof of Theorem \ref{th: property star} follows then from the construction given in the proof of Theorem \ref{th: complete mp}, where $B$ is taken to be the empty graph (similarly to \cite{KS17}), and in $A_1$ we take $H_{A_1}$ to be the graph guaranteed by Lemma \ref{l: ABlemma}, as detailed in the above paragraph. Again, there can be a small, yet no negligible amount of edges in $G[B]$ such that their endpoints have small common degree in $A_1$. For these edges, we define $B_2$ to be the set of vertices which are not $A_1$-good in $B$, and we use $A_2$ and $A_3$ in the same manner. We note that here Lemma~\ref{lemma:H_B_1-subgraph} is not relevant any more.

\section{Discussion and open problems}\label{s: discussion}
In this paper, we present Conjecture \ref{conjecture} and make progress towards resolving it, in particular obtaining a universal bound of $O(n\ln n)$ for the saturation number $\textit{sat}(G(n,p),F)$ for all $F$ (Theorem \ref{th: global 1}), and characterising a family of graphs for which the bound is linear in $n$ (Theorem \ref{th: global 2}). On the same matter, let us also iterate the question raised in the introduction:
\begin{question}
    Is it true that, for all constant $0<p<1$ and all graphs $F$ where every edge of $F$ is in a triangle, \textbf{whp} $\textit{sat}\left(G(n,p),F\right)=\left(1+o(1)\right)n\log_{\frac{1}{1-p}}n$?
\end{question}

For specific graph families, we answer positively. In particular, we extended the sharp asymptotics results of \cite{KS17}, both to a wide family of graphs (Theorem \ref{th: property star}), and to complete multipartite graphs when $p\ge \frac{1}{2}$ (Theorem \ref{th: complete mp}). However, our proof of Theorem \ref{th: complete mp} requires our graph to be dense enough, so that we may find a subgraph $A$ of an appropriate size and such that the neighbourhoods of the vertices outside this subgraph form a dense enough subset in a Hamming space with the domain $A$. 
 Unfortunately, this no longer holds when $p$ is a small constant (with the technicalities explicit in Lemma \ref{l: many neighbours}). It would be interesting to know whether a similar construction, with different probabilistic or combinatorial tools, could extend the result for $p<\frac{1}{2}$.

Finally, let us reiterate that it was shown in \cite{DHKZ23} that Lemma \ref{l: ABlemma} is, in fact, tight, in the sense that the conditions for it are both sufficient and necessary. This implies that the results that can be obtained by the construction given in \cite{KS17}, where $B$ is taken as an empty graph, are limited to those of Theorem \ref{th: property star}. Indeed, already trying to extend this result to complete multipartite graphs, which may not adhere to the conditions of Lemma \ref{l: ABlemma}, required a delicate (and at times technical) treatment, utilising a coupling with an auxiliary random graph in a Hamming space and proving a tight bound for its independence number using a covering-balls argument.

\paragraph{Acknowledgements} The authors wish to thank the anonymous referees for their helpful comments, and in particular for their corrections in Section \ref{s: problematic section}. The authors further wish to thank Michael Krivelevich for helpful comments and suggestions, and Dor Elboim for fruitful discussions.

\bibliographystyle{abbrv}
\bibliography{main}
\end{document}